\documentclass[11pt,twoside]{amsart}
\usepackage{amsmath, amsthm, amscd, amsfonts, amssymb, graphicx, color}
\usepackage[bookmarksnumbered, colorlinks, plainpages]{hyperref}

\theoremstyle{plain}
\newtheorem{thm}{Theorem}[section]

\newtheorem{lem}[thm]{Lemma}

\theoremstyle{remark}
 
\numberwithin{equation}{section}
\def\pn{\par\noindent}

\newcommand{\al}{\alpha}

\newcommand{\Z}{\mathbb{Z}}

\newcommand {\la}{\langle}
\newcommand {\ra}{\rangle}

\newcommand{\be}{\beta}

\begin{document}

\leftline{ \scriptsize \it International Journal of Group Theory  Vol. {\bf\rm XX} No. X {\rm(}201X{\rm)}, pp XX-XX.}

\vspace{1.3 cm}

\title{Groups of order 2048 with three generators and three relations}
\author{S. FOULADI and  R. ORFI$^*$}

\thanks{{\scriptsize
\hskip -0.4 true cm MSC(2010): Primary: 20F05 ; Secondary: 20D15.
Received: 26 July 2011, Accepted: 21 June 2010.\\
$*$Corresponding author
\newline\indent{\scriptsize $\copyright$ 2011 University of Isfahan}}}

\maketitle

\begin{center}
Communicated by\;
\end{center}

\begin{abstract} It is shown that there are exactly seventy-eight 3-generator
2-groups of order $2^{11}$ with trivial Schur multiplier. We then
give 3-generator, 3-relation presentations for forty-eight of them
proving that these groups have deficiency zero.
\end{abstract}

\vskip 0.2 true cm


\pagestyle{myheadings}
\markboth{\rightline {\scriptsize FOULADI and  ORFI}}
         {\leftline{\scriptsize Groups of order 2048 with three generators and three relations }}

\bigskip
\bigskip


\section{\bf Introduction}
\vskip 0.4 true cm

A finite group is said to have deficiency zero if it has a
deficiency zero presentation, namely a presentation with an equal
number of generators and relations. A classical fact is that
finite groups of deficiency zero have trivial Schur multiplier,
for example see [\ref{J}, p.87]. So the Schur multiplier provides
a useful criterion in the search for finite groups of deficiency
zero. But the converse is not true since there are many examples
of finite groups with trivial Schur multiplier and non-zero
deficiency. These groups are all non-nilpotent. In fact, it is a
long-standing question about finite $p$-groups  with trivial Schur
multiplier whether they have deficiency zero, see [\ref{W},
Question 12]. In [\ref{G}], the authors prove a number of
$p$-groups have deficiency zero and give explicit presentations
for them with an equal number of generators and relations. It is
noted in [\ref{G}] that there are no $3$-generator $2$-groups of
order less than $2^9$ having trivial Schur multiplier and there
exist exactly two such groups of order $2^9$. Moreover in
[\ref{F}] we see that there are exactly eighteen $3$-generator
$2$-groups of order $2^{10}$ with trivial Schur multiplier all
having deficiency zero. Many finite
 $d$-generator, $d$-relation groups are known for $d=1,2,3$. Trivial
examples are the finite cyclic groups with $d=1$ and the symmetric
group of degree 3 with $d=2$. In fact many examples with $d=2$
have been given by several authors. Examples of finite groups with
$d=3$ are infrequent, see [\ref{Ja}] and the references therein.
It might be worth noting that there are no known examples of
finite groups with $d=4$ and finite nilpotent $4$-generator groups
require at least $5$ defining relations by
 a celebrated theorem of Golod-Shafarevich.
 Such groups have been constructed in [\ref{GN}, \ref{HN}], of
 orders
 $2^{14},
 2^{16}, 2^{17}, 2^{18}$ and $2^{19}$.

 In this paper, using computational methods we show that there
 are exactly seventy-eight $3$-generator $2$-groups of order $2^{11}$
 with trivial Schur multiplier. We then give $3$-generator,
 $3$-relation presentations for forty-eight of them proving that these groups
 all have deficiency zero.

\vspace{0.3cm}
 Our notation is  standard. $\mathbb{Z}_{n}$ is the cyclic group of order $n$.
  The direct product of $\ell$
copies of $\Z_{n}$ is denoted by $ \Z_{n}^\ell$. The Schur
multiplier of the group $G$ is denoted by $M(G)$. We write
SmallGroup$(n,m)$ for the $m$th
 group of order $n$ as quoted in the "Small Groups" library in \textsf{GAP} [\ref{GA}].


\section{\textbf{Method}}

\vspace*{0.4cm} In this section our first step is to determine all
$3$-generator groups of order $2^{11}$ with trivial Schur
multiplier. Then our second step is to show that some of
these groups have deficiency zero.\\
We describe below a method that enables one to determine
$3$-generator groups of order $2^{11}$ having trivial Schur
multiplier.  We use the computer algebra systems \textsf{GAP}
[\ref{GA}] and {\sc Magma} [\ref{MA}] which contain a data library
\lq \lq Small Groups\rq \rq   providing access to the descriptions
of the groups of order at most $2000$ except $2^{10}$, prepared by
Besche {\it et al} [\ref{BEO}]. Following [\ref{G}], our main
strategy is to determine some particular extensions, called
descendants, of specified $3$-generator $2$-groups $G$, $|G|\leq
2^9$, in the hope of finding $2$-groups of order $2^{11}$ with
trivial Schur multiplier. To do this we will use the following
theorems.

\begin{thm}\label{2.1} \textup{[\ref{K}, Theorem 3.2.1]}
Suppose that $N$ is a normal subgroup of a finite group $G$. If
$F$ is a free group of finite rank, $R$ is a normal subgroup of
$F$ for which $G\cong F/R$ and $S$ is a normal subgroup of $F$ for
which $SR/R$ corresponds to $N$, then there is an exact sequence
$$1\rightarrow(\frac{R\cap[F,S])}{([F,R]\cap[F,S])}\rightarrow
M(G)\rightarrow M(\frac{G}{N})\rightarrow \frac{(N\cap
G')}{[N,G]}\rightarrow1.$$
\end{thm}

\begin{thm}\label{2.2} \textup{[\ref{K}, Corollary 3.2.2]} Suppose that
$N$ is a normal subgroup of a finite group $E$. If $M(E)=1$, then
$M(\frac{E}{N})\cong \frac{(N\cap E')}{[N, E]}.$
\end{thm}

Recall that the lower exponent-$p$ central series of $G$ is a
descending series of subgroups defined  recursively by
$P_{0}(G)=G$,  $P_{i+1}(G)=[P_{i}(G), G] P_{i}(G)^{p}$ for $i\geq
0$. If $c$ is the smallest integer such that $P_{c}(G)=1$, then
$G$ has exponent-$p$ class $c$. A group $E$ is said to be a
descendant of a finite $d$-generator $p$-group $G$ with
exponent-$p$ class $c$ if the quotient $E/P_{c}(E)$ is isomorphic
to $G$. A group is called an immediate descendant of $G$ if it is
a descendant of $G$ and has exponent-$p$ class $c+1$. Both
\textsf{GAP} and {\sc Magma} compute the lower exponent-$p$
central series of a finite group using the $p$-quotient algorithm
described in [\ref{MB}] and are able to construct all immediate
descendants of a given $p$-group by the $p$-group generation
algorithm [\ref{O}].

Now we determine all $3$-generator $2$-groups $E$ of order
$2^{11}$ with trivial Schur multiplier.
\begin{lem}\label{2.3}
Let $E$ be a $3$-generator $2$-group of order $2^{11}$ with
trivial Schur multiplier. Then $E$ is an immediate descendant of a
$3$-generator group $G$ of order $2^n$ $(n\leq 10)$ which
satisfies $M(G)\cong \Z_{2}^\ell$, where \break $0\leq \ell \leq
11-n$.
\end{lem}

\begin{proof}
Suppose that $E$ has exponent-$p$ class $c+1.$ Using Theorem
\ref{2.2}, with $N=P_{c}(E)$, we have $M(E/P_{c}(E))\cong
(P_{c}(E)\cap E')/[ P_{c}(E), E].$ By our hypothesis on the class
of $E$, we observe that $P_{c}(E)^2=1$, from which we conclude
that $P_{c}(E)$ is an elementary abelian $2$-group and that
$M(E/P_{c}(E))\cong P_{c}(E)\cap E'$. Now the group $G:=E/P_{c}(
E)$ is a $3$-generator group with $M(G)\hookrightarrow P_{c}(E)$
and so $|M(G)|\leq |E|/|G|$.
\end{proof}

The above lemma reduces the number of groups that need to be
considered
 dramatically. We use {\sc Magma} and \textsf{GAP} to construct all immediate descendants $E$ of
 such groups $G$ and rule out those having non-trivial Schur
 multiplier. Since all groups $G$ of order $2^n$ $(n\leq 9)$ are available in
\textsf{GAP}, first we determine all immediate descendants of
groups $G$ which satisfy $M(G)\cong \Z_{2}^\ell$, where $0\leq
\ell \leq 11-n$. In the list below there are forty $3$-generator
groups of order $2^{11}$ with trivial Schur multiplier with the
above property. We use the notation $[n, m, k]$ for the group $E$,
where $E$ is the $k$th immediate descendant of the group
$G$=SmallGroup$(n, m)$.
 \vspace*{0.2cm}

\noindent [512, 6489, 2], [512, 6489, 3], [512, 6490, 2], [512,
6490, 3], [512, 9113, 4],\break [512, 9113, 5], [512, 9114, 4],
[512, 9114, 5], [512, 9121, 4], [512, 9121, 5],\break [512, 9122,
4], [512, 9122, 5], [512, 9137, 4], [512, 9137, 5], [512, 9146,
4], \break [512, 9146, 5] , [512, 12397, 4], [512, 12397, 5],
[512, 12398, 4],\break [512, 12398, 5], [512, 12399, 4], [512,
12399, 5], [512, 12400, 4],\break [512, 12400, 5], [512, 12401,
4], [512, 12401, 5], [512, 12402, 4],\break [512, 12402, 5], [512,
12403, 2], [512, 12403, 3], [512, 12404, 2],\break [512, 12404,
3], [512, 12413, 4], [512, 12413, 5], [512, 12414, 4],\break [512,
12414, 5], \hspace{.2 cm}[512, 12423, 4], [512, 12423, 5], [512,
12424, 4], \break [512,
12424, 5].\\

Now by Lemma \ref{2.3}, we have to consider 3-generator groups $G$
of order $2^{10}$ with $M(G)\cong \Z_{2}^\ell$, where $0\leq \ell
\leq 1$. All 3-generator groups of order $2^{10}$ with trivial
Schur multiplier are classified in [\ref{F}]. By using
\textsf{GAP} we see that there is no immediate descendant of order
$2^{11}$ of these eighteen groups of order $2^{10}$. Since groups
of order $2^{10}$ are not available in \textsf{GAP}, we state the
following theorem to construct groups of order $2^{10}$ with Schur
multiplier of order 2.

\begin{thm}\label{2.4}
Let $E$ be a $3$-generator $2$-group of order $2^{10}$ with
$M(E)\cong \Z_{2}$. Then $E$ is an immediate descendant of a
$3$-generator group $G$ of order $2^n$ with $n\leq 9$ which
satisfies either $M(G)\cong \Z_{2}^\ell$,  $0\leq \ell \leq 11-n$
or
 $M(G)\cong \Z_{4}\times \Z_{2}^\ell$, $0\leq \ell \leq 9-n$.
\end{thm}

\begin{proof}
Suppose that $E$ has exponent-$p$ class $c+1.$ By Theorem
\ref{2.1}, we have the following exact sequence: $1\rightarrow
(R\cap[F,S])/([F,R]\cap[F,S])\rightarrow M(E)\xrightarrow{\al}
M(E/P_{c}(E))\xrightarrow{\be} (P_{c}(E)\cap
E')/[P_{c}(E),E]\rightarrow 1,$ where $F$ is a free group of
finite rank, $R$ is a normal subgroup of $F$ for which $E\cong
F/R$ and $S$ is a normal subgroup of $F$ for which $SR/R$
corresponds to $P_{c}(E)$. On setting $G=E/P_{c}(E)$ we see that
$M(G)/Ker{\be}\cong P_{c}(E)\cap E'$ and $M(E)/Ker{\al}\cong
Ker{\be}$. Therefore $Ker{\be}=1$ or $Ker{\be}\cong \Z_{2}$ since
$M(E)\cong \Z_{2}$. Now since $M(G)/Ker{\be}\hookrightarrow
P_{c}(E)$ and $P_{c}(E)$ is elementary abelian, we deduce that
$|M(G)|\leq 2|P_{c}(E)|$ and $M(G)$ is either elementary abelian
or $M(G)\cong \Z_{4}\times \Z_{2}^\ell$.
\end{proof}

Now it only remains to determine all immediate descendants of
groups of order $2^{10}$ with Schur multiplier of order 2. In the
list below there are thirty-eight groups of order $2^{11}$ with
trivial Schur multiplier with the above property. We use the
notation $[n, m, k, t]$ for the group $E$, where $E$ is the $t$th
immediate descendant of the group $L$ such that $L$ is the $k$th
immediate descendant of $G$=SmallGroup$(n, m)$, in fact $L$ is a
3-generator group of order $2^{10}$ with Schur multiplier of order
2.
 \vspace*{0.2cm}

\noindent [512, 53479, 3, 1], [512, 53479, 3, 2], [512, 53480, 1,
1], [512, 53480, 1, 2],\break [512, 53480, 2, 1], [512, 53480, 2,
2], [256, 2525, 8, 1], [256, 2525, 8, 2],\break [256, 2525, 9, 1],
[256, 2525, 9, 2], [256, 2528, 5, 1], [256, 2528, 5, 2],\break
[256, 2528, 6, 1], [256, 2528, 6, 2], [256, 3638, 8, 1], [256,
3638, 8, 2],\break [256, 3639, 8, 1], [256, 3639, 8, 2], [256,
3640, 5, 1], [256, 3640, 5, 2],\break [256, 3640, 6, 1], [256,
3640, 6, 2], [256, 3641, 8, 1], [256, 3641, 8, 2],\break [256,
3641, 9, 1], [256, 3641, 9, 2], [256, 3641, 10, 1], [256, 3641,
10, 2],\break [256, 3643, 8, 1], [256, 3643, 8, 2], [256, 3643, 9,
1], [256, 3643, 9, 2],\break
 [256, 3643, 10, 1], [256, 3643, 10,
2],  [256, 2522, 6, 1],  [256, 2522, 6, 2],
 \break  [256, 2523, 6,
1], [256, 2523, 6, 2]. \vspace*{0.2cm}

 The second step is to give
$3$-generator, $3$-relation
 presentations for the groups obtained in the first step. We used
 mainly the method described in [\ref{HN}] to find such a
 presentation for each group $G$ under consideration. Our
 first attempt towards obtaining such presentations for $G$ was to
 find several triples  of generators for each group. On each
 generating triple, we computed a presentation $\la X | R \ra$ using
  the  relation finding algorithm of Cannon [\ref{C}] which is available
 in \textsf{GAP} and {\sc Magma}. Then an attempt was made to find a subset $S$ of $R$
 having three elements such that $\la X | S\ra $ defines $G$. In
 searching for generating triples for each group a small set of
 group elements was chosen  by a knowledge of conjugacy classes
 and checked for generating triples. This technique was also used
 in [\ref{F}] to determine deficiency zero presentations for all
 3-generator, 2-groups of order $2^{10}$ with trivial Schur
 multiplier. The authors obtained seventeen deficiency zero presentations
 from eighteen groups in [\ref{F}] by this method. It seems that
 this method is useful to find deficiency zero
 presentations. Moreover in this paper to find the order of the
 groups defined by the presentations $\la X | S\ra $ as above, we use Knuth-Bendix algorithm
   in {\sc KBMAG}  package, which was  written by Derek Holt [\ref{GA}]. By the above observation we show
 that forty-eight groups from seventy-eight 3-generator groups of
 order $2^{11}$ with trivial Schur multiplier, have deficiency zero.

\section{\textbf{Results}}

\vspace*{0.2cm} In three tables below we list all $3$-generator
$2$-groups of order $2^{11}$ with trivial Schur multiplier. In
tables 1 and 2 we list forty-eight groups with deficiency zero.
Also table 3 give a presentation for the remaining thirty groups
with more than three relations in which we show that the above
method failed to find a balanced presentation for these groups.
Entries of the form $[n, m, k]$ and $[n, m, k, t]$ were described
in the previous section. An attempt was made to choose a
presentation for each group with a reasonably small length.

\hspace{4.5cm} \noindent \textbf{ Table 1 }

\noindent \begin{tabular}{ l  l  l  l  }

\hline
 Group No. & Relators & $[n, m, k]$
\\ \hline

$\#1$ & $b^{-1}acabc^{-1}$, $ c^2ab^2a $, $ab^{-1}cacb^{-1}$ & $[ 512, 6489, 2 ]$ \\

$\#2$ &  $bca^{-1}c^{-1}ba$, $bac^{-1}b^{-3}ca$, $ cbca^{-1}c^2a^{-1}b$ & $[ 512, 6489, 3 ]$  \\

$\#3$ &  $ba^{-1}c^{-2}ba$, $ a^2cbcb^{-1}$, $ bcbaca^{-1}$ & $[ 512, 6490, 2 ]$\\

$\#4$ & $bac^{-1}b^{-1}ca$, $ b^2cac^{-1}a$, $b^{-1}cbc^3a^{-2}$ & $[ 512, 6490, 3 ]$\\

$\#5$ & $b^2ca^{-1}ca$, $c^{-1}a^2b^{-1}cb$, $ ba^{-1}b^{-1}c^4a$ & $[512, 9113, 4 ]$\\

$\#6$ & $ba^{-1}b^3a$, $b^{-1}c^3bc^{-1}$, $a^3b^{-1}cba^{-1}c^{-1}$ & $[ 512, 9113, 5 ]$\\

$\#7$ & $ bca^{-1}b^{-1}ca$,  $b^3cbc$,  $a^3b^{-1}a^{-1}cbc^{-1}$ & $[512, 9114, 4 ]$\\

$\#8$ & $ cbcb^{-1}$,  $ba^{-1}b^3a$,  $a^3c^3a^{-1}c^{-1}$ & $[ 512, 9114, 5 ]$\\

$\#9$ & $ ba^{-1}c^{-1}bc^{-1}a$,  $ab^{-2}cac^{-1}$,  $a^2c^3bc^{-1}b^{-1}$ & $[ 512, 9122, 4 ]$\\

$\#10$ & $  ba^{-1}c^{-1}bc^{-1}a$, $ca^{-1}c^{-1}b^2a^{-1}$,  $a^2c^{-1}bc^3b^{-1}$ & $[ 512, 9122, 5 ]$\\

$\#11$ & $ a^{-1}cbcab^{-1}$, $b^3a^{-1}ba$, $ cac^3a^{-3}$ & $[ 512, 9137, 4 ]$\\

\end{tabular}\\

\hspace{4.5cm} \noindent \textbf{ Table 1 }

\noindent \begin{tabular}{ l  l  l  l  }

\hline
 Group No. & Relators & $[n, m, k]$
\\ \hline

$\#12$ & $ a^{-1}cbcab^{-1}$, $b^3a^{-1}ba$, $ a^3ca^{-1}c^3$ & $[ 512, 9137, 5 ]$\\

$\#13$ & $ cbcb^{-1}$, $a^{-1}ba^{-2}b^2a^{-1}b$, $c^3aca^{-3}$ & $[512, 9146, 4 ]$\\

$\#14$ & $ cbcb^{-1}$, $ a^{-1}ba^{-2}b^2a^{-1}b$, $a^3c^3a^{-1}c$ & $[ 512, 9146, 5 ]$\\

$\#15$ & $a^3bab, acab^{-1}c^{-1}b, b^2cac^{-3}a$ & $[ 512, 12397, 4]$\\

$\#16$ & $ a^3bab, acab^{-1}c^{-1}b, b^2c^{-1}ac^3a$  & $[ 512,12397, 5 ]$\\

$\#17$ & $ baba^{-1}, a^2ca^{-1}cbab^{-1}, b^3c^3b^{-1}c^{-1}$ & $[ 512, 12398, 4 ]$ \\

$\#18$ & $cac^{-1}a, b^{-1}a^{-1}ba^2b^2a, bac^{-3}bca$ & $[ 512, 12398, 5 ]$\\

$\#19$ & $ bab^{-1}a, a^{-1}c^3ac^{-1}, a^3c^{-1}abc^{-1}b^{-3}$ & $[512, 12400, 4 ]$\\

$\#20$ & $bab^{-1}a, a^{-1}c^3ac^{-1}, a^2b^{-1}a^{-1}ca^{-1}bcb^{-2}$ & $[ 512, 12400, 5 ]$\\

$\#21$ & $cbc^{-1}b, ac^3a^{-1}c^{-1}, ba^{-1}ba^2b^2a$ & $[512, 12401, 4 ]$\\

$\#22$ & $cbc^{-1}b, a^{-1}c^3ac^{-1}, ba^{-1}ba^2b^2a$ & $[ 512, 12401, 5 ]$\\

$\#23$ & $cbc^{-1}b, ba^{-1}b^3a^{-1}, a^3cac^{-3}$ & $[ 512, 12402, 4]$\\

$\#24$ & $ab^{-3}ab^{-1}, bc^{-3}bc^{-1}, a^2bca^{-1}ca^{-1}b^{-1}$ & $[512, 12404, 2 ]$\\

$\#25$ & $bab^3a, cbc^3b, bacacb^{-1}a^{-2}$ & $[ 512, 12404, 3 ]$\\

$\#26$ & $ba^{-1}bc^2a, ac^2bab^{-1}, (ca)^2cbcb^{-1}$ & $[ 512, 12413, 5 ]$\\

$\#27$ & $b^{-1}a^{-1}b^3a, ca^{-1}cbab^{-1}, a^2c^{-2}(bc^{-1})^2$ & $[ 512, 12414, 4 ]$\\

$\#28$ &  $cbac^{-1}ab^{-1}, acbca^{-1}b^{-1}, a^2c^{-1}a^{-1}bc^{-1}ab$ & $[ 512, 12424,4 ]$\\

\end{tabular}\\

 \hspace{4.5cm} \noindent \textbf{ Table 2 }

\noindent \begin{tabular}{ l  l  l  l  }

\hline
 Group No. & Relators & $ [n, m, k, t]$
\\ \hline

$\#29$ &  $cbc^{-1}b, a^3b^{-1}c^{-1}ba^{-1}c, a^3bc^{-1}a^{-1}c^{-1}b $ & $[ 512, 53480, 2, 1 ]$\\

$\#30$ &  $cbc^{-1}b, a^{-3}bcacb, a^3c^{-1}a^{-1}cb^2 $  & $[  512, 53480, 2, 2 ]$ \\

 $\#31$ & $ cbcb^{-1}, ba^{-1}c^{-1}bca, a^2b^{-1}acbac $ & $ [ 256, 2528, 5, 1 ]$\\

$\#32$  & $a^2c^2, bcb^{-1}ac^{-1}a, bacb^3ca^{-1} $ & $[ 256, 2528, 5, 2 ]$ \\

$\#33$ & $ a^2(bc)^2, bcb^{-1}aca^{-1}, a^3c^{-2}bab^{-1} $  & $[256, 2528, 6, 1 ]$ \\

$\#34$ & $ ba^{-1}cb^{-1}ca, cb^2ca^{-2}, a^3bcba^{-1}c^{-1} $ &  $[ 256, 2528, 6, 2 ]$\\

 $\#35$  & $c^{-1}bcb, ac^{-1}ac^{-1}b^2, a^2ca^{-1}bc^{-1}ab$ & $ [ 256, 3640, 5, 1 ]$ \\

$\#36$ &  $ ba^{-1}cbc^{-1}a, ac^3a^{-1}c^{-1}, b^{-3}c^{-1}bca^2 $ & $[ 256, 3640, 5, 2 ]$ \\

$\#37$ & $ a^{-1}b^3ab^{-1}, b^{-1}a^{-1}cabc, c^{-1}ac^{-3}bab$ & $ [ 256, 3640, 6, 1 ]$ \\

\end{tabular}\\

\hspace{4.5cm} \noindent \textbf{ Table 2 }

\noindent \begin{tabular}{ l  l  l  l  }

\hline
 Group No. & Relators & $ [n, m, k, t]$
\\ \hline

$\#38$ & $ a^{-1}b^3ab^{-1}, b^{-1}a^{-1}cabc, c^{-2}ac^{-1}bcab $ & $[ 256, 3640, 6, 2 ]$ \\

$\#39$ & $ a^{-1}cba^{-1}c^{-1}b, b^3cb^{-1}c^{-1}, a^3b^{-1}c^2ab^{-1} $ & $[ 256, 3641, 10, 1 ]$ \\

$\#40$ & $ caca^{-1}, a^{-1}b^3ab^{-1}, acabc^{-3}b $ & $[ 256, 3641, 10, 2 ]$\\

$\#41$ & $ cac^3a, bcba^{-1}b^{-2}ca, babca^{-1}c^{-1}a^{-2} $ & $[ 256, 3643, 8, 1 ]$ \\

$\#42$ & $ ac^{-3}ac^{-1}, bcba^{-1}b^{-2}ca, babca^{-1}c^{-1}a^{-2}$ & $ [ 256, 3643, 8, 2 ]$ \\

$\#43$ & $ ac^{-3}ac^{-1}, a^{-1}cbcb^2a^{-1}b^{-1}, babca^{-1}c^{-1}a^{-2} $ & $[ 256, 3643, 10, 1 ]$ \\

$\#44$ & $ cac^3a, b^3cb^{-1}a^{-1}ca, babca^{-1}c^{-1}a^{-2} $ & $ [ 256, 3643, 10, 2 ]$\\

 $\#45$ &  $a^2(ab^{-1})^2, bc^3b^{-1}c^{-1}, a^3b^2cac^{-1} $  & $[ 256, 2522, 6, 1 ]$\\

 $\#46$ & $ a^2(ab^{-1})^2, bc^3b^{-1}c^{-1}, a^3c^{-1}ab^2c $ & $[ 256, 2522, 6, 2 ]$ \\

$\#47$  &  $ ca^{-1}bca^{-1}b^{-1}, cab^{-2}ca^{-1}, a^2c^2(ba)^2 $ & $ [ 256, 2523, 6, 1 ]$\\

$\#48$ & $a^2(ab^{-1})^2, b^{-1}c^3bc^{-1}, a^3cabcb $ & $[ 256,
2523, 6, 2 ]$\\

\end{tabular}\\

\hspace{4.5cm} \noindent \textbf{ Table 3 }

\noindent \begin{tabular}{ l  l  l  l  }

\hline
  Group No. & Relators & $ [n, m, k, t]$
\\ \hline

$\#49$ & $ bac^{-1}b^{-1}ca, ba^{-1}b^3a, ab^5a^{-1}b^{-1}, a^2bcb^3c^{-1},   $ & $ [ 512, 9121, 4 ]$ \\
       & $cabc^{-3}ba$ & \\

$\#50$ & $ bac^{-1}b^{-1}ca, ba^{-1}b^3a, ab^5a^{-1}b^{-1}, a^2bcb^3c^{-1},  $ & $[ 512, 9121, 5 ] $ \\
       & $ c^{-1}bc^{-2}acab$ & \\

$\#51$ & $a^3b^{-1}a^{-1}b, ac^3a^{-1}c^{-1}, a^5bab^{-1}, a^2c^2a^{-2}c^{-2},  $ & $ [ 512, 12399, 4 ]$\\
       & $ aca^{-1}bcb^{-3}$ & \\

 $\#52$ &  $bc^3b^{-1}c^{-1}, b^3a^3b^{-1}a^{-1}, aca^{-1}c^{-1}a^{-1}cac^{-1},  $  & $[ 512, 12399, 5 ]$\\
        & $ ab^{-1}c^{-1}a^{-1}cb^3, a^2cac^{-2}ac $ &  \\

$\#53$  &  $ bab^3a, b^2cb^2c^{-1}, a^{-1}c^3ac^{-1}, bc^{-1}bca^{-4},$ & $ [ 512, 12402, 5 ]$\\
        & $(bc)^2(bc^{-1})^2 $ & \\

$\#54$ & $  a^3b^{-1}a^{-1}b, a^5bab^{-1}, a^3c^{-1}ac^3,   $ & $[ 512, 12403, 2 ] $ \\
       & $a^2c^2a^{-2}c^{-2}, babc^{-1}b^{-1}cba$ & \\

$\#55$ & $ a^3b^{-1}a^{-1}b, a^5bab^{-1}, a^3c^{-3}ac,  $ & $ [ 512, 12403, 3 ]$ \\
       & $a^2c^2a^{-2}c^{-2}, bab^2cb^{-1}c^{-1}a $ & \\

\end{tabular}

\vspace{1cm}

 \hspace{4.5cm} \noindent \textbf{ Table 3 }

\noindent \begin{tabular}{ l  l  l  l  }

\hline
  Group No. & Relators & $ [n, m, k, t]$
\\ \hline

$\#56$ & $ ba^{-1}bc^2a, ac^2bab^{-1}, b^2cb^2c^{-1}, $ & $[ 512, 12413, 4 ] $\\
       & $ cacb^{-1}c^{-1}bc^{-1}a$ & \\

 $\#57$ &  $ bcb^{-1}ca^{-2}, a^2b^{-1}cbc, a^{-1}c^3ac^{-1},$  & $[ 512, 12414, 5 ]$\\
        & $ a^3b^{-3}a^{-1}b $ &  \\

$\#58$  &  $ ac^2ab^2,  c^2a^2b^{-2},  a^{-1}cbcab, $ & $[ 512, 12423, 4 ] $\\
        & $ c^{-2}ac^{-1}bca^{-1}b,  cbaca^{-1}bc^{-1}a^{-1}c^{-1}ba^{-1}b^{-1}$ & \\

$\#59$  &  $ac^{-1}abc^{-1}b^{-1},  b^{-1}cbaca,  cbc^{-3}b,  $ & $[ 512, 12423, 5 ] $\\
        & $ a^3b^{-1}cac^{-1}b^{-1} $ & \\

$\#60$  &  $bcbc^{-1}a^{-2},  b^2c^{-1}aca^{-1},  b^{-1}ca^{-1}cba,  $ & $ [ 512, 12424, 5 ]$\\
        & $ a^2c^3bcb $ & \\

$\#61$  &  $ b^2ab^2a^{-1},  c^2bc^2b^{-1},  (bc)^2ac^{-1}a^{-1}c,  $ & $[ 512, 53479, 3, 1 ] $\\
        & $ cbcaba^{-3},  a^3cb^{-1}a^{-1}bc^{-1},  a^2c^7b^{-1}c^{-1}b$ & \\

$\#62$  &  $  a^2ca^{-2}c^{-1}, bacbc^{-1}a, bca^{-1}c^{-1}ba,$ & $ [ 512, 53479, 3, 2 ]$\\
        & $acabc^{-1}b^{-1}, c^2bc^2b^{-1}, b^{16}ca^{-2}c $ & \\

$\#63$  &  $ cab^2c^{-1}a^{-3}, cba^{-1}cb^{-1}a^3, a^2c^{-1}b^{-1}a^{-1}cab, $ & $[ 512, 53480, 1, 1 ] $\\
        & $ b^{-2}cbaba^{-1}c, a^{-1}(bc)^2ac^2$ & \\

$\#64$  &  $ c^{-1}a^{-1}c^{-1}bab, b^{-1}c^2aba^{-1}, a^2ba^2b^{-1}, $ & $ [ 512, 53480, 1, 2 ]$\\
        & $ a^2ca^2c^{-1}, bc^3bca^2c^2$ & \\

$\#65$  &  $ c^4, a^2ba^2b^{-1}, b^{-1}cacba, acab^{-1}c^{-1}b, $ & $[ 256, 2525, 8, 1 ] $\\
        & $ b^5c^{-1}a^{-1}cb^{-1}a$ & \\

$\#66$  &  $ c^4, a^2ba^2b^{-1}, b^{-1}cacba, acab^{-1}c^{-1}b,  $ & $ [ 256, 2525, 8, 2 ]$\\
        & $b^5cb^{-1}aca^{-1} $ & \\

$\#67$  &  $ bcac^{-1}ba, b^2cb^{-2}c^{-1}, a^3bc^{-1}bac^{-1}, $ & $[ 256, 2525, 9, 1 ] $\\
        & $ a^2c^{-1}bc^3b^{-1}, a^3cab^{-1}cb^{-1}$ & \\

$\#68$  &  $ c^{-1}b^{-2}cb^2, bc^{-1}acba, b^{-1}acbca^{-3}, $ & $ [ 256, 2525, 9, 2 ]$\\
        & $c^2b^{-1}a^{-1}bc^2a, a^3b^{-1}cb^{-1}ac $ & \\

$\#69$  &  $ ac^{-2}a^{-1}c^2, ac^{-1}b^{-1}abc, acbac^{-1}b^{-3},$ & $ [ 256, 3638, 8, 1 ]$\\
        & $ c^{-1}a^2cba^2b^{-1}, a^3ca^{-1}b^{-2}c$ & \\

\end{tabular}

\vspace{1cm}

\hspace{4.5cm} \noindent \textbf{ Table 3 }

\noindent \begin{tabular}{ l  l  l  l  }

\hline
  Group No. & Relators & $ [n, m, k, t]$
\\ \hline

$\#70$  &  $ b^{-1}c^{-1}acba, bc^2b^{-1}c^{-2}, a^4c^2b^{-2},$ & $[ 256, 3638, 8, 2 ] $\\
        & $a^3cb^{-1}cab $ & \\

$\#71$  &  $a^2ca^2c^{-1}, ba^{-1}cb^{-1}ca, b^2cb^{-2}c^{-1},  $ & $ [ 256, 3639, 8, 1 ]$\\
        & $ bc^2b^{-1}c^{-2}, a^3b^{-1}ca^{-1}cb^{-1}, b^6c^2$ & \\

$\#72$  &  $a^2ca^2c^{-1}, bacb^{-1}ca^{-1}, b^2cb^{-2}c^{-1}, $ & $ [ 256, 3639, 8, 2 ]$\\
        & $bc^2b^{-1}c^{-2}, a^{-1}bc^{-1}ba^3c, b^2c^6 $ & \\

$\#73$  &  $ a^2ca^2c^{-1}, b^{-1}cbaca^{-3}, cab^2ca^{-3},$ & $[ 256, 3641, 8, 1 ] $\\
        & $ a^3b^{-1}c^2ab^{-1}$ & \\

$\#74$  &  $ a^2ca^2c^{-1}, b^{-1}cbaca^{-3}, cab^2ca^{-3},$ & $[ 256, 3641, 8, 2 ] $\\
        & $ba^{-1}c^2ba^{-3} $ & \\

$\#75$  &  $a^2b^{-2}, a^2cb^2c^{-1}, a^{-1}c^3ac^{-1}, $ & $[ 256, 3641, 9, 1 ] $\\
        & $ (bc)^2(b^{-1}c^{-1})^2, acbc^{-1}a(b^{-1}a^{-1})^2b^{-1}$ & \\

$\#76$  &  $ a^2b^2, a^2cb^{-2}c^{-1}, a^{-1}c^3ac^{-1},$ & $[ 256, 3641, 9, 2 ] $\\
        & $ (bc)^2(b^{-1}c^{-1})^2, babc^{-1}aba^{-1}cb^{-1}a^{-1}$ & \\

$\#77$  &  $ a^2ca^2c^{-1}, ab^2a^{-1}b^{-2}, b^3cb^{-1}ca^{-2}, $ & $, [ 256, 3643, 9, 1 ] $\\
        & $ bcacb^{-1}a^{-3}, bacbc^{-3}a^{-1}$ & \\

$\#78$  &  $ a^2ca^2c^{-1}, ab^2a^{-1}b^{-2}, a^2cb^{-1}cb^3, $ & $[ 256, 3643, 9, 2 ]  $\\
        & $ bcacb^{-1}a^{-3}, bac^{-1}bc^3a^{-1}$ & \\

\end{tabular}\\

\begin{center}{\textbf{Acknowledgments}}
\end{center}
The authors are grateful to the referees for their  valuable
suggestions.  The work of authors was in part supported by  Arak University. \\



\bigskip
\bigskip


{\footnotesize \pn{\bf  Authors:}\; \\ {Department of
Mathematics}, {University
of Arak, P.O.Box 38156-88349,} {Arak, Iran}\\
{\tt Email: s-fouladi@araku.ac.ir}\\
{\tt Email: r-orfi@araku.ac.ir}\\

\end{document}